\documentclass[11pt]{article}
\usepackage[T1]{fontenc}
\usepackage{latexsym,amssymb,amsmath,amsfonts,amsthm}
\usepackage{graphicx}
\usepackage{subfigure}
\usepackage{placeins}
\usepackage{color}
\usepackage{indentfirst}
\numberwithin{equation}{section}
\newcommand{\R}{{\mathbb R}}

\newcommand{\N}{{\mathbb N}}
\newcommand{\C}{{\mathbb C}}
\newcommand{\s}{{\mathbb S}}

\newcommand{\be}{\begin{eqnarray}}
\newcommand{\ben}{\begin{eqnarray*}}
\newcommand{\en}{\end{eqnarray}}
\newcommand{\enn}{\end{eqnarray*}}

\newcommand{\curl}{{\rm curl\,}}

\newcommand{\divv}{{\rm div\,}}

\newtheorem{theorem}{Theorem}[section]
\newtheorem{lemma}[theorem]{Lemma}
\newtheorem{corollary}[theorem]{Corollary}

\newtheorem{remark}[theorem]{Remark}

\definecolor{rot}{rgb}{1.000,0.000,0.000}

\begin{document}
\renewcommand{\theequation}{\arabic{section}.\arabic{equation}}
\begin{titlepage}
\title{\bf Unique determination of a penetrable scatterer of rectangular type for
inverse Maxwell equations \\ by a single incoming wave} 


\author{
Guanghui Hu\thanks{ Beijing Computational Sciences Research Center, Beijing 100193, China ({\tt
hu@csrc.ac.cn}).}
\and
Long Li\thanks{LMAM, School of Mathematical Sciences, Peking University, Beijing 100871,China, and
Beijing Advanced Innovation Center for Imaging Technology, Capital Normal University, Beijing 100048, China ({\tt 1601110046@pku.edu.cn}).
}
\and Jun Zou
\thanks{Department of Mathematics, The Chinese University of Hong Kong, Hong Kong.
({\tt zou@math.cuhk.edu.hk}).}
 }
\date{}
\end{titlepage}
\maketitle
\vspace{.2in}

\begin{abstract} This work is concerned with an inverse electromagnetic scattering problem in two dimensions. We prove that in the TE polarization case, the knowledge of the electric far-field pattern incited by a single incoming wave is sufficient to uniquely determine the shape of a penetrable scatterer of rectangular type.
As a by-product,
the uniqueness is also confirmed to inverse transmission problems
modelled by scalar Helmholtz equations with discontinuous normal derivatives at the scattering interface.


\vspace{.2in} {\bf Keywords}: Uniqueness, inverse medium scattering, Maxwell equations, one incoming wave, shape identification, right corners.
\end{abstract}

\section{Introduction and main results}\label{sec:introduction}
Assume a time-harmonic electromagnetic plane wave $(E^{in}, H^{in})$ is incident onto an infinitely long penetrable scatterer of cylindrical type sitting inside a homogeneous background medium. We use $D\subset \R^2$ to denote
the cross-section of the scatterer in the $ox_1x_2$-plane so that the space occupied by the scatterer can be represented
by $\Omega=D\times \R\subset \R^3$. We assume that $D$ is a bounded open domain with the connected exterior $D^e:=\R^2\backslash\overline{D}$. The medium inside the scatterer is supposed to be homogeneous, isotropic and invariant along the $x_3$-direction.
The direction of the incoming wave is supposed to be perpendicular to the $ox_3$-axis, i.e., ${\rm \textbf{d}}=(d,0)$ where $d=(d_1, d_2)\in \s:=\{x=(x_1, x_2): |x|=1\}$. In the TE (transverse electric) polarization case, the incident electric plane wave takes the form
\ben
E^{in}({\rm \textbf{x}})=(u^{in}(x), 0), \quad\quad {\rm \textbf{x}}=(x, x_3)\in\R^3,
\enn
where
\[
u^{in}(x)=d^\perp
\,e^{i\kappa x\cdot d},\quad d^\perp:=(-d_2, d_1)\,,
\]
and $\kappa>0$ denotes the wavenumber of the homogeneous background medium.
The total field $E$ can be analogously written as $E({\rm \textbf{x}})=(u(x), 0)$, where
$u=(u_1, u_2)$ is governed by the two-dimensional Maxwell system
\be\label{eq:Helm}
\overrightarrow{\nabla} \times (\nabla \times u)-\kappa^2  q(x) u =0\quad\mbox{in}\quad \R^2,
\en
where $q$ is the refractive index. 
In this paper we assume that the function $q$, after normalization, is given by the piecewise constant function
\ben
q(x)=\left\{\begin{array}{cll}
1,\quad&\mbox{if}\quad x\in D^e,\\
q_0\neq 1,\quad&\mbox{if}\quad x\in \overline{D}.
\end{array}\right.
\enn
Note that $\overrightarrow{\nabla}$ and $\nabla$ are two dimensional {\rm curl} operators defined by
\be
\overrightarrow{\nabla}\times f=(\partial_2f, -\partial_1 f),\qquad \nabla\times u=\partial_1u_2-\partial_2u_1.
\en
Let $\nu\in \s:=\{x\in\R^2: |x|=1\}$ be the unit normal on $\partial D$ pointing into $D^e$,
and $\tau:=\nu^\bot$. Then it is direct to derive
the interface conditions complemented to the system \eqref{eq:Helm} by using
the continuity of the tangential components of $E$ and $\curl E$ across the interface $\partial D$:
\be\label{TE}
\nabla\times u^+=\nabla\times u^-,\quad \tau\cdot u^+=\tau\cdot u^-\quad\mbox{on}\quad \partial D,
\en
where the superscripts $\pm$ stand for the limits taken from outside and inside of $D$, respectively.
 At the infinity, the scattered field
 $u^{sc}:=u-u^{in}$ is assumed to meet
 the two-dimensional Silver-M\"{u}ller radiation condition (see also \cite{BP2012, Zou2015})
\begin{equation}\label{eq:radiation}
\lim_{|x|\rightarrow \infty} \sqrt{|x|}\,\left\{ \nabla\times u^{sc}-i\kappa u^{sc}\cdot \hat{x}^\bot \right\}=0,
\end{equation}
uniformly in all directions $\hat{x}=x/|x|\in\s$.
In particular, the radiation condition (\ref{eq:radiation}) leads to  the asymptotic expansion
\be\label{far-field}
u^{sc}(x)=\frac{e^{i\kappa r}}{\sqrt{r}}\left\{ u^\infty(\hat{x})\,\hat{x}^\bot+O(1/r) \right\}\quad\mbox{as}\quad r=|x|\rightarrow\infty,
\en
uniformly in all $\hat{x}\in \s$.
 The function $u^\infty(\hat x)$ in \eqref{far-field}
 is analytically defined on $\s$, and referred to as the electric \emph{far-field pattern} or the \emph{scattering amplitude} in the TE polarization case, where vector $\hat{x}\in\s$ is called the observation direction of the 
 far field.

It is worth noting that (\ref{eq:radiation}) is a reduction of the three-dimensional
Silver-M\"{u}ller radiation condition
\ben
\lim_{|{\rm \textbf{x}}|\rightarrow \infty} |{\rm \textbf{x}}|\,\left\{ (\nabla\times E^{sc})\times \hat{{\rm \textbf{x}}}-i\kappa E^{sc} \right\}=0,
\enn
to two dimensions, where $E^{sc}({\rm \textbf{x}})=(u^{sc}(x),0)$ and
\ben
{\rm \textbf{x}}=(x,0),\quad \;\hat{{\rm \textbf{x}}}={\rm \textbf{x}} /|{\rm \textbf{x}}|\in \s^2:=\{{\rm \textbf{x}}\in \R^3: |{\rm \textbf{x}}|=1\}.
\enn
The asymptotic expansion (\ref{far-field}) implies that the far-field patterns $(E^\infty, H^\infty)$ of the scattered electromagnetic fields $(E^{sc}, H^{sc})$ take the form
\ben
E^\infty(\hat{\textbf{x}})=u^\infty(\hat{x})\,\begin{pmatrix}
\hat{x}^\bot \\ 0
\end{pmatrix},\quad
H^\infty(\hat{\textbf{x}})=u^\infty(\hat{x}) \begin{pmatrix}
0 \\ 0 \\ 1
\end{pmatrix}.
\enn
Obviously, there hold the relations that
\ben
\hat{\textbf{x}}\cdot E^\infty(\hat{\textbf{x}})=0,\qquad
\hat{\textbf{x}}\cdot H^\infty(\hat{\textbf{x}})=0, \qquad
H^\infty(\hat{\textbf{x}})=\hat{\textbf{x}}\times E^\infty(\hat{\textbf{x}}).
\enn
Hence,
the knowledge of the far-field data $u^\infty(\hat{x})$ is equivalent to that of the electric far-field pattern measured on $\s\times\{0\}\subset\s^2$.

The inverse scattering problem of our interest is to recover the interface $\partial D$
from the far-field pattern $u^\infty(\hat{x})$ for all $\hat{x}\in\s$, and the main result we shall establish in this work
can be stated below.
\begin{theorem}\label{TH}
The boundary $\partial D$ of an arbitrary rectangle $D$ in $\R^2$ is
uniquely determined by the far-field pattern $u^\infty(\hat{x})$ for all $\hat{x}\in\s$ incited
by a single incoming wave.
\end{theorem}
Theorem \ref{TH} indicates that, in the TE polarization case, the far-field data of a single incoming wave is sufficient to determine the shape of a rectangular penetrable scatterer in $\R^2$. This is a global uniqueness result within the class of penetrable rectangles for the 2D Maxwell system. For the local uniqueness,
we have the following result.
\begin{corollary}\label{Cor1.2}
Let $\Omega=D\times\R$ and $\tilde{\Omega}=\tilde{D}\times \R$ be two infinitely long penetrable scatterers of cylindrical type, and $u^\infty$ and $\tilde{u}^\infty$ be their respective far-field patterns (TE case).
If $\partial D$ differs from $\partial \tilde{D}$ in the presence of a right corner lying on the boundary of the unbounded component of $\R^2\backslash\overline{D\cup\tilde{D}}$ (see Figure \ref{F2}). Then it is impossible that $u^\infty(\hat{x})=\tilde{u}^\infty(\hat{x})$ for all $\hat{x}\in \s$.
\end{corollary}
\begin{figure}[htbp]
\centering
 \scalebox{0.3}{\includegraphics{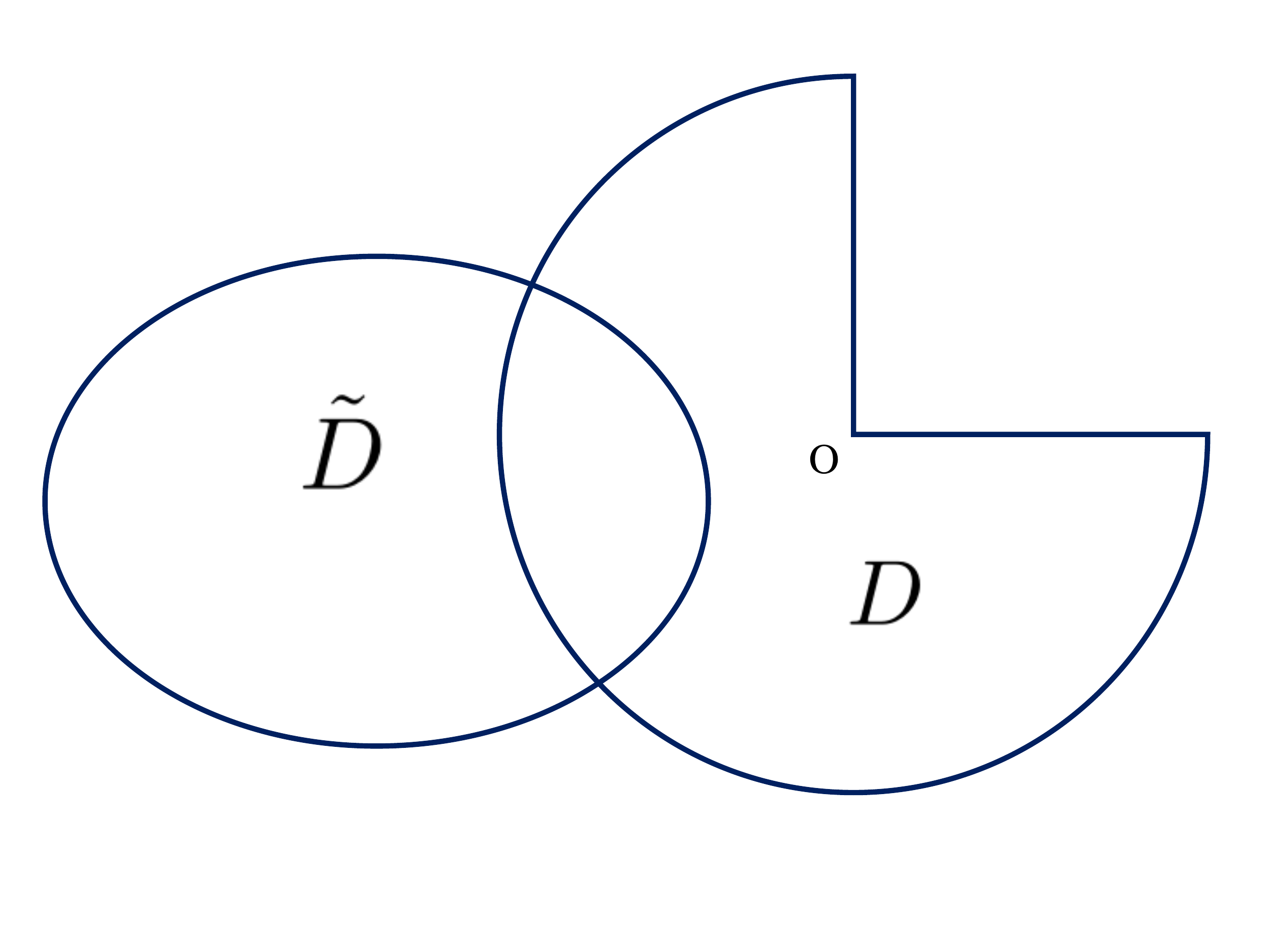}}
\caption{\label{F2}  Penetrable scatterers $D$ and $\tilde{D}$ cannot produce identical outgoing waves (far-field patterns) due to the right corner $O$ lying on the boundary of the unbounded component of $\R^2\backslash\overline{D\cup\tilde{D}}$. }
\end{figure}

As we shall see from the proof of Theorem \ref{TH}, a right corner scatters any incoming electric wave non-trivially
in the TE case, as stated in the following corollary.
\begin{corollary}\label{Cor1.3} Let $\Omega=D\times\R$  be an infinitely long penetrable scatterer of cylindrical type and $u^{in}$ be any incoming wave such that
\ben
\Delta u^{in}+\kappa^2  u^{in}=0,\quad \divv u^{in}=0\qquad \mbox{in a neighbourhood of $D$}.
\enn
If $\partial D$ possesses a right corner, then the scattered field $u^{sc}$
 corresponding to the transmission problem \eqref{eq:Helm}-\eqref{TE}
 cannot vanish identically in $D^e$. \end{corollary}

 It is easy to observe that, in the TE polarization case, the magnetic field is of the form 
$H(\textbf{x})=(0, 0, v(x))$, where the scalar function $v:=1/(i\kappa)\nabla\times u$ is governed by the Helmholtz equation
\be\label{TE-H}
\Delta v+\kappa^2  q\; v=0\quad\mbox{in}\quad \R^2,
\en together with the transmission conditions
\be\label{TE-Ht}
v^+=v^-,\quad \partial_\nu v^+=\lambda\; \partial_\nu v^-\quad\mbox{on}\quad \partial D.
\en
Here $\lambda:=1/q_0\neq 1$ and $v=v^{in}+v^{sc}$ in $\R^2\backslash\overline{D}$, with
 \ben
 v^{in}(x)=1/(i\kappa)\nabla\times u^{in}(x)=-e^{i\kappa x\cdot d},\quad  v^{sc}(x)=1/(i\kappa)\nabla\times u^{sc}(x).
 \enn
 One can derive from (\ref{eq:radiation}) that $v^{sc}$ fulfills the two-dimensional Sommerfeld radiation condition
 \ben
 \lim_{r\rightarrow \infty} \sqrt{r}\,\left( \partial_r v^{sc}-i\kappa v^{sc} \right)=0,\quad r=|x|.
 \enn
 In the current paper we do not consider the equivalent scalar transmission problem
(\ref{TE-H})-(\ref{TE-Ht}), because the subsequent analysis
for the two-dimensional Maxwell equation \eqref{eq:Helm}-(\ref{TE}) would provide us more insights
into a possible approach for treating the full Maxwell system in three dimensions.
The 3D case appears much more challenging than the planar case, and is
still an open fundamental problem in inverse medium scattering problems.
On the other hand, the transmission conditions in (\ref{TE}) keep the continuity of the Cauchy data $(\nabla\times u, \tau\cdot u)$ of the 2D Maxwell system across the interface $\partial D$.
These conditions can be more easily handled by our approach than
the transmission conditions (\ref{TE-Ht}) with $\lambda\neq 1$.
The transmission problem (\ref{TE-H})-(\ref{TE-Ht}) with $\lambda=1$
corresponds to the TM (transverse magnetic) polarization of our scattering problem, for which
the electric and magnetic fields are of the form
\ben
E({\rm \textbf{x}})=(0,0, v(x)),\quad H({\rm \textbf{x}})=(u_1(x),u_2(x),0),\quad {\rm \textbf{x}}=(x,0)\in \R^3.
\enn
We refer to \cite{BLS,HSV, ElHu2015, EH2017,PSV} for recent studies on inverse transmission problems of the scalar Helmholtz equation with $\lambda=1$ (i.e., the refractive index $q$ is continuous or has no jumps
across the interface between the homogeneous and inhomogeneous media), not only in two dimensions but also in higher dimensions. It was shown that, under the assumption $\lambda=1$, the global and local uniqueness results
as stated in Theorem \ref{TH} and Corollary \ref{Cor1.2} remain valid for curvilinear polygonal and polyhedral scatterers with variable potential functions (see \cite{EH2017}), and that one-dimensional interfaces with even a "weakly" singular point always scatter incoming waves non-trivially (see \cite{LHY}).
The arguments in these references do not apply to our current case of $\lambda\neq 1$, because the jumps of the normal derivatives would bring essential difficulties.
To the best of our knowledge, no uniqueness results with a single incoming wave are available for
$\lambda\neq 1$ and piecewise constant $q(x)$.
It is worth noting that the case $q\equiv q_0$ in $\R^2$ (that is, the wave speed remains constant in the whole space) was verified by Ikehata in \cite{Ikehata} for convex penetrable polygonal obstacles, as a byproduct of the enclosure method. In this paper we investigate the more practical case that $q$ is a piecewise constant function.
Our results apply automatically to the Helmholtz system (\ref{TE-H})-(\ref{TE-Ht}), since it is equivalent to the two dimensional Maxwell system \eqref{eq:Helm}-(\ref{TE}).
A comparison of the essential differences between our
arguments and results with the ones in the recent paper \cite{LX2017} on electromagnetic corner scattering will be made in Remark \ref{Rem} of Section\,\ref{sec:2}.

\begin{remark}
The unique solvability of the scattering problem \eqref{eq:Helm}, \eqref{TE} and \eqref{eq:radiation}
follows from that of (\ref{TE-H})-(\ref{TE-Ht}). In fact, applying the integral equation method or the variational approach one can prove that  the scalar problem (\ref{TE-H})-(\ref{TE-Ht}) admits a unique solution $v\in H^1_{loc}(\R^2)$ (see e.g., \cite{CK} and \cite[Chapter 5]{Cakoni}). This implies that the original scattering  problem \eqref{eq:Helm}, \eqref{TE} and \eqref{eq:radiation} has a unique solution
 \ben
 u\in X:=\{u\in L^2_{loc}(\R^2)^2,\quad \nabla\times u\in  H^1_{loc}(\R^2)\}.
 \enn
 Since $q$ is piecewise constant, by elliptic interior regularity (see \cite{Gil-Tru})  it is easy to see that $u$ is analytic in both $D$ and $D^e$.
\end{remark}

Let us also mention that, if $\Omega\subset \R^3$ is a bounded perfectly conducting obstacle of polyhedral type, its geometrical shape  $\partial\Omega$ can be uniquely determined by a single electric far-field pattern $E^{\infty}( \hat{{\rm x}})$ for all $\hat{{\rm x}}\in \s^2$; we refer to
\cite{Zou07, Zou09} for the analysis based on the reflection principle for the full Maxwell system.
For general penetrable and impenetrable scatterers, uniqueness in shape identification and medium recovery can be proved if the far-field patterns for all incident directions and polarization vectors with a fixed wavenumber are available; see \cite[Chapter 7.1]{CK}, \cite{Hahner}, \cite{Kress1, Kress2, Sun1992} and references therein.

The remaining part of this paper is organized as follows. In the subsequent Section \ref{sec:2}, we prove Theorem \ref{TH} and Corollaries \ref{Cor1.2} and \ref{Cor1.3}, all of which are essentially built upon the results in
Lemma \ref{Lem:1}. The proof of Lemma \ref{Lem:1} will be given in Section \ref{sec:proof} via induction arguments.

\section{Proofs of the main results}\label{sec:2}
In this section we prove our main results stated in the previous section, namely,
Theorem \ref{TH} and Corollaries \ref{Cor1.2}-\ref{Cor1.3}.
For this, we need a fundamental auxiliary result, whose proof is postponed
to Section \ref{sec:proof}.
\begin{lemma}\label{Lem:1}
Let $B=\{x: |x|<1\}$, $\Gamma=\{(x_1,0): x_1\geq 0\} \cup \{(0,x_2): x_2\geq 0\}$. Then for
any two distinct constants $q_1$ and $q_2$ in $\C$,
the solutions $u$ and $v$ to the vector-valued Helmholtz equations
\begin{equation}\label{eq:transmission}
\left\{ \begin{array}{lll}
\Delta u+q_1 u=0 &&\mbox{in}\quad B,\\
\Delta v+q_2 v=0 &&\mbox{in}\quad B,\\
{\rm div}\,u={\rm div}\,v=0&&\mbox{in}\quad B, \\
\tau\cdot u=\tau\cdot v && \mbox{on}\quad B\cap \Gamma,\\
\nabla\times u=\nabla\times v && \mbox{on}\quad B\cap\Gamma
\end{array}
\right.
\end{equation}
vanish identically in the unit ball $B$.
\end{lemma}

Lemma \ref{Lem:1} shows a local property of the two-dimensional Maxwell system around a domain with
a right corner: the Cauchy data of such two Maxwell equations cannot coincide on an interface with
a right corner, if the wavenumbers involved are not identical.

With the help of Lemma \ref{Lem:1}, we can now establish our main results in Section \ref{sec:introduction}.
We will provide a detailed proof of Theorem \ref{TH}, but omit the proofs of
Corollaries \ref{Cor1.2} and \ref{Cor1.3} as they can be verified basically in the same manner as it is done
below for Theorem \ref{TH}.

To prove Theorem \ref{TH}, we assume that $D$ and $\tilde{D}$  are the cross-sections of two infinitely long rectangular penetrable scatterers whose scattered fields are denoted by $u^{sc}$ and $\tilde{u}^{sc}$, respectively.
Analogously, we
let $u^\infty$ and $\tilde{u}^\infty$ be the far-field patterns of $u^{sc}$ and $\tilde{u}^{sc}$.
Supposing that  $u^\infty(\hat{x})=\tilde{u}^\infty(\hat{x})$ for all $\hat{x}\in \s$, we need to show $D=\tilde{D}$.
Assume on the contrary that $D\neq\tilde{D}$, then we derive a contradiction below.

We first apply Rellich's lemma (see \cite{CK}) to obtain
\[
u^{sc}=\tilde{u}^{sc}\quad\mbox{in}\quad D^e\cap \tilde{D}^e.
\]
If $D\neq \tilde{D}$, one can always find a right corner $O$ and a small number $\epsilon>0$ such that either
 $O\in\partial D$ and $B_\epsilon(O)\subset \tilde{D}^e$, or $O\in\partial \tilde{D}$ and  $B_\epsilon(O)\subset D^e$. Without loss of generality, we suppose that the former case holds;
 see Figure \ref{F1}. Note that if $D\cap\tilde{D}=\emptyset$, one can easily derive a contradiction by extending the scattered field to the whole space and then applying Rellich's lemma.

\begin{figure}[htbp]
\centering
 \scalebox{0.3}{\includegraphics{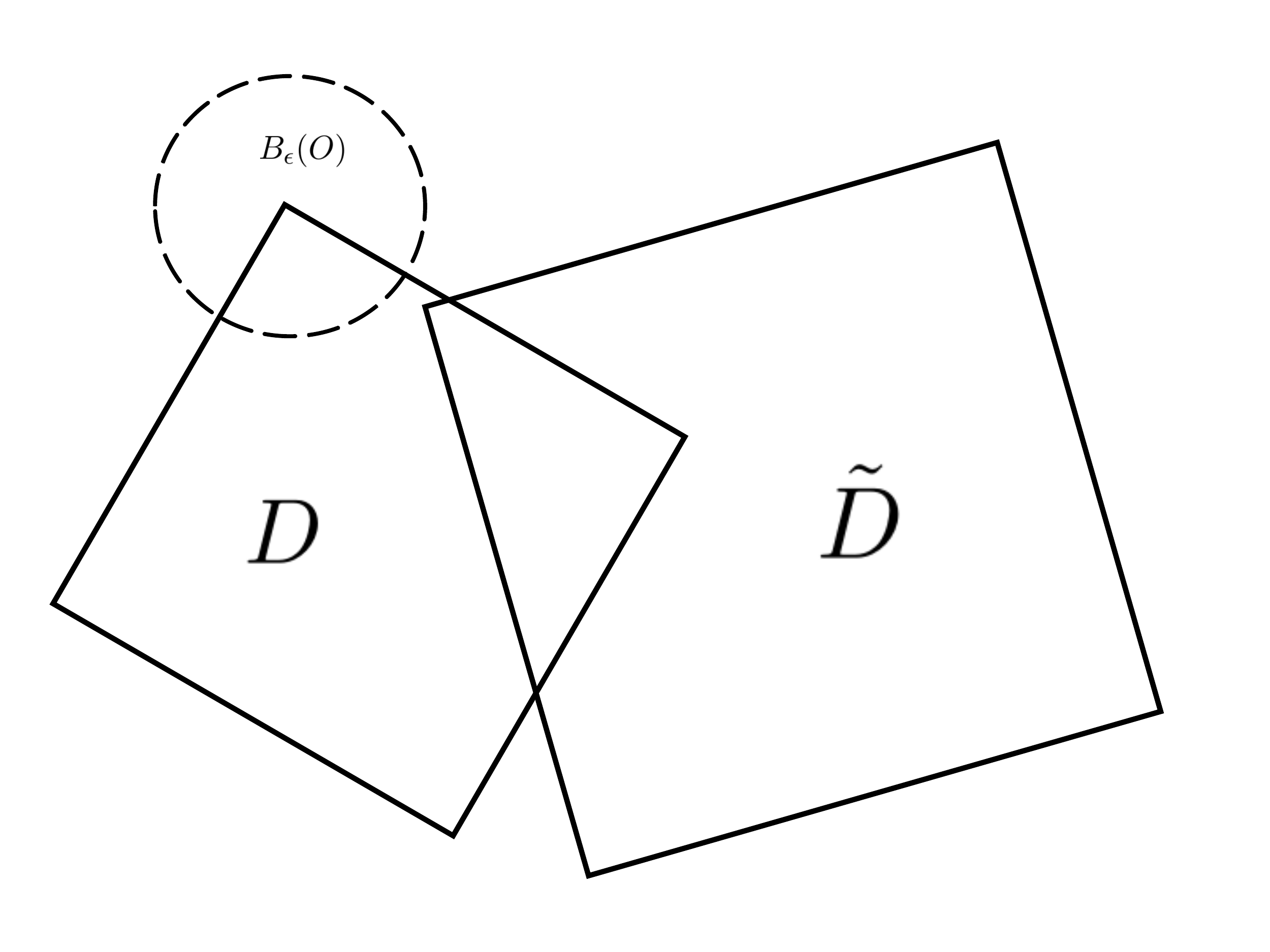}}
\caption{\label{F1}  Two rectangular penetrable scatterers with the same far-field data.}
\end{figure}

Since the Maxwell system  is invariant by both rotation and translation, we may suppose, without loss of generality,
that the corner $O$ coincides with the origin and
that the two sides of this right corner lie on the positive $ox_1$ and
$ox_2$ axes, respectively. Then we have the following coupled system
\ben
\left\{ \begin{array}{lll}
\overrightarrow{\nabla} \times (\nabla \times u)-\kappa^2  q_0 u=0\quad&&\mbox{in}\quad B_\epsilon(O)\cap D,\\
\overrightarrow{\nabla} \times (\nabla \times \tilde{u})-\kappa^2  \tilde{u}=0\quad&&\mbox{in}\quad B_\epsilon(O),\\
\tau\cdot u=\tau\cdot \tilde{u} &&\mbox{on}\quad  B_\epsilon(O)\cap \partial D,\\
\nabla\times u= \nabla\times\tilde{u}&&\mbox{on}\quad  B_\epsilon(O)\cap \partial D.
\end{array}\right.
\enn
Since $\tilde{u}$ is analytic in $B_\epsilon(O)$ and the interface $B_\epsilon(O)\cap\partial D$  is piecewise analytic,
 the Cauchy pair $(\tau\cdot u, \nabla\times u)$ is piecewise analytic on
the boundary $B_\epsilon(O)\cap\partial D$.
Recalling the Cauchy-Kovalevskaya theorem, one may extend $u$ analytically from
 $\overline{D}\cap B_\epsilon(O)$ to a small neighborhood of the corner $O$ in
 the exterior domain
$D^e\cap B_\epsilon(O)$.  For notational convenience, we still denote by $B_\epsilon(O) $ the extended domain. Further, the extended function, which we still denote by $u$,  satisfies the Maxwell equation
\ben
\overrightarrow{\nabla} \times (\nabla \times u)-\kappa^2  q_0 u=0\quad\mbox{in}\quad B_\epsilon(O).
\enn


Using the relation that $\overrightarrow{\nabla} \times (\nabla \times u)=-\Delta u + \nabla (\nabla\cdot u)$,
we may apply Lemma \ref{Lem:1} to $u$ and $\tilde{u}$  with $q_1=\kappa^2q_0$, $q_2=\kappa^2$ to deduce that
both $u$ and $\tilde{u}$ vanish in $B_\epsilon(O)$, where we have used
the assumption that $q_0\neq 1$. By the unique continuation of the Helmholtz equation, we have $u\equiv 0$ in $\R^2$.
This is a contradiction to the fact that $|u^{in}|=1$ in $\R^2$ and $|u^{sc}|$ decays as $|x|$ tends to infinity. Hence, we have shown that $D=\tilde{D}$, and complete the proof of Theorem \ref{TH}.

\begin{remark}\label{Rem}
We think that it is possible to prove Lemma \ref{Lem:1} for the full Maxwell system in a cuboid domain in three dimensions. This was verified in \cite{LX2017} by using the CGO solutions
for an admissible set of electric and magnetic fields which contain both electromagnetic planar waves and dipoles.
Using this result,
it was proved in \cite{LX2017} that the scattered fields do not vanish
within the admissible set.
But we emphasize that this result can not lead to any uniqueness result to
the inverse scattering problem with a single incoming wave.
In this work, we are able to demonstrate that any solution to
the transmission problem of the two dimensional Maxwell system of the form
in Lemma \ref{Lem:1} must vanish.
This excludes the inadmissible set considered in \cite{LX2017} for the TE case,
hence help us establish the unique identifiability for the inverse medium scattering
with a single incoming wave. Our studies show that the uniqueness issue is more difficult
than the corner scattering problems that try to justify the non-vanishing of the scattered fields.
\end{remark}

\section{Proof of Lemma \ref{Lem:1}}\label{sec:proof}
We shall make full use of the expansion of the solutions to
the Helmholtz equation in the Cartesian coordinate system to prove Lemma \ref{Lem:1}.
We find the expansion in Cartesian coordinates
particularly convenient for verifying Lemma \ref{Lem:1} in domains with a right corner.
The expansion in Cartesian coordinates was proved to be an effective approach
for the scalar Helmholtz equation \cite{ElHu2015}, but we encountered
essential technical difficulties in our efforts to apply this approach
to the current vector-valued Helmholtz equations.

Since $u$ and $v$ are the solutions to the Helmholtz equation with constant potentials, the functions $u$ and $v$ are both analytic in $B$. Hence $u$ and $v$ can be expanded in Taylor expansion:
\begin{eqnarray}
&&u=\sum_{n\ge0} \sum_{m \ge0} U_{n,m}\, x_1^n x_2^m
:=\sum_{n\ge0} \sum_{m \ge0} (u_{n,m}^{(1)}, u_{n,m}^{(2)})\, x_1^n x_2^m\,,
\label{eq:Unm} \\
&&v=\sum_{n\ge0} \sum_{m \ge0} V_{n,m}\, x_1^n x_2^m
:=\sum_{n\ge0} \sum_{m \ge0} (v_{n,m}^{(1)}, v_{n,m}^{(2)})\, x_1^n x_2^m
\label{eq:Vnm}
\end{eqnarray}
%
for $U_{n,m}=(u_{n,m}^{(1)}, u_{n,m}^{(2)})\in \C^2$ and
$V_{n,m}=(v_{n,m}^{(1)}, v_{n,m}^{(2)})\in \C^2$. By plugging the expansions into the
Helmholtz equations in \eqref{eq:transmission}, we easily derive the relations satisfied by
$U_{n,m}$ and $V_{n,m}$:
\be\label{eq:1}\begin{split}
&(n+1)(n+2) U_{n+2,m}+(m+1)(m+2)U_{n,m+2}+q_1 U_{n,m}=0,\\
&(n+1)(n+2) V_{n+2,m}+(m+1)(m+2)V_{n,m+2}+q_2 V_{n,m}=0.
\end{split}
\en
Let $A_{n,m}:=U_{n,m}-V_{n,m}=(a_{n,m}^{(1)}, a_{n,m}^{(2)})$.
Then it is easy to see the difference $w=u-v$ admits the Taylor expansion
\be\label{eq:anm}
w=\sum_{n\ge0} \sum_{m \ge0} A_{n,m}\, x_1^n x_2^m
=\sum_{n\ge0} \sum_{m \ge0} (a_{n,m}^{(1)}, a_{n,m}^{(2)})\, x_1^n x_2^m
\quad\mbox{in}\quad B
\en
and satisfies the equation
\be\label{eq:w}
\Delta w+ q_1 w=(q_2-q_1)v\quad\mbox{in}\quad B.
\en

We first derive some important relations for the coefficients $a_{n,m}^{(j)}$ in \eqref{eq:anm}.
\begin{lemma}\label{lem:all_anm}
For $j=1,2$ and all $n,m \geq 0$, the coefficients $a_{n,m}^{(j)}$ in \eqref{eq:anm} satisfy
\begin {equation}
\begin{split}
&(m+4)(m+3)(m+2)(m+1)a^{(j)}_{n,m+4}
+(n+4)(n+3)(n+2)(n+1)a^{(j)}_{n+4,m}\\
&+2(n+2)(n+1)(m+2)(m+1)a^{(j)}_{n+2,m+2}
+(q_1+q_2)(n+2)(n+1)a^{(j)}_{n+2,m}\\
&+(q_1+q_2)(m+2)(m+1)a^{(j)}_{n,m+2}
+q_2q_1 a^{(j)}_{n,m}=0\,,
\end{split}
\label{eq:3}
\end{equation}
%
%
\begin{equation}
(n+1)a^{(1)}_{n+1,m}=-(m+1)a^{(2)}_{n,m+1}\,,
\label{eq:4}
\end{equation}
\begin{equation}
a^{(1)}_{n,0}=0,\quad  a^{(2)}_{0,m}=0\,,
\label{eq:5}
\end{equation}
\begin{equation}
(n+1)a^{(2)}_{n+1,0}=a^{(1)}_{n,1}, \ a^{(2)}_{1,m}=(m+1)a^{(1)}_{0,m+1}\,.
\label{eq:6}
\end{equation}
\end{lemma}

\begin{proof}
Using the expression \eqref{eq:anm} and
the equation \eqref{eq:w}, we derive by direct computing that
the coefficients $a^{(j)}_{n,m}$ fulfill the recursive relations
\ben
(n+1)(n+2) a^{(j)}_{n+2,m}+(m+1)(m+2)a^{(j)}_{n,m+2}+q_1 a^{(j)}_{n,m}=(q_2-q_1)v^{(j)}_{n,m}
\enn
for $j=1,2$, or equivalently, the coefficients $v^{(j)}_{n,m}$ satisfy that for all $n, m\geq 0$,
\be\label{eq:2}
v^{(j)}_{n,m}=\frac{1}{q_2-q_1}\left[(n+1)(n+2) a^{(j)}_{n+2,m}+(m+1)(m+2)a^{(j)}_{n,m+2}+q_1 a^{(j)}_{n,m}\right].
\en
Then the desired results in \eqref{eq:3} follow by
inserting (\ref{eq:2}) in the second equation of (\ref{eq:1}).
On the other hand, the relations \eqref{eq:4} can follow
directly from the divergence-free condition of $w=u-v$ (see \eqref{eq:transmission}).
Now using the transmission conditions that
$\tau\cdot w=\nabla\times w=0 \ \mbox{on} \ B\cap \Gamma$ (see \eqref{eq:transmission}),
we can obtain the desired relations \eqref{eq:5}-\eqref{eq:6}.
\end{proof}

In the rest of this section we shall establish the desired results in Lemma \ref{Lem:1} by proving
$a^{(1)}_{n,m}=a^{(2)}_{n,m}=0$ for all $n, m\in \N^+$.
Then it follows from (\ref{eq:2}) that $v_{n,m}^{(j)}=0$ for all $n, m\in \N^+$ and $j=1,2$, leading to $v\equiv 0$ by analyticity.
Analogously, the vanishing of $u$ follows from the relation
\ben
u^{(j)}_{n,m}=\frac{1}{q_2-q_1}\left[(n+1)(n+2) a^{(j)}_{n+2,m}+(m+1)(m+2)a^{(j)}_{n,m+2}+q_2 a^{(j)}_{n,m}\right].
\enn
Our proof is essentially based on an sophisticated induction argument on
$M:=n+m=0, 1, 2, \cdots$, making full use of  the relations (\ref{eq:3})-(\ref{eq:6}).

It is easy to observe that there are a total of $2(M+1)$ coefficients $a^{(j)}_{n,m}$ ($j=1,2$) for each fixed
$M=n+m$, while (\ref{eq:4})-(\ref{eq:6}) give
$M+2$ linear relations with $M-2$ unknown coefficients;
see the diagram below where the line segment means a linear relation between the entries at
two ends of the segment:
\ben
&&\begin{pmatrix}
a_{M,0}^{(1)} \\
a_{M,0}^{(2)}
\end{pmatrix}\diagdown
\begin{pmatrix}
a_{M-1,1}^{(1)} \\
a_{M-1,1}^{(2)}
\end{pmatrix}\diagdown
\begin{pmatrix}
a_{M-2,2}^{(1)} \\
a_{M-2,2}^{(2)}
\end{pmatrix}\diagdown\cdots\diagdown
\begin{pmatrix}
a_{2,M-2}^{(1)} \\
a_{2,M-2}^{(2)}
\end{pmatrix}\diagdown
\begin{pmatrix}
a_{1,M-1}^{(1)} \\
a_{1,M-1}^{(2)}
\end{pmatrix}\diagdown
\begin{pmatrix}
a_{0,M}^{(1)} \\
a_{0,M}^{(2)}
\end{pmatrix},\\
&&\begin{pmatrix}
a_{M,0}^{(1)} \\
a_{M,0}^{(2)}
\end{pmatrix}\diagup
\begin{pmatrix}
a_{M-1,1}^{(1)} \\
a_{M-1,1}^{(2)}
\end{pmatrix},\qquad\qquad\qquad\qquad\qquad\qquad\qquad\quad
\begin{pmatrix}
a_{1,M-1}^{(1)} \\
a_{1,M-1}^{(2)}
\end{pmatrix}\diagup
\begin{pmatrix}
a_{0,M}^{(1)} \\
a_{0,M}^{(2)}
\end{pmatrix}.
\enn

Our proof is divided in the following several lemmas.

\begin{lemma}\label{lem:step2}
The coefficients $a_{n,m}^{(j)}$ in \eqref{eq:anm} satisfy that
$a^{(1)}_{n,m}=a^{(2)}_{n,m}=0$ for $n+m=0,1, 2$.
\end{lemma}

\begin{proof}
By setting $n,m=0,1$ in (\ref{eq:5}) we obtain
\be\label{step1}
 a^{(1)}_{0,0}=a^{(1)}_{1,0}=a^{(2)}_{0,0}=a^{(2)}_{1,0}=0.
\en
This, together with (\ref{eq:4}) and (\ref{eq:6}), gives $a^{(1)}_{0,1}=a^{(2)}_{0,1}=0$.
Then taking $n,m=2$ in (\ref{eq:5}), we see $a^{(1)}_{2,0}=a^{(2)}_{0,2}=0$.
Further, by setting $(n,m)=(0,1), (1,0)$ in (\ref{eq:4}), respectively, we derive
\[
a^{(1)}_{1,1}=-a_{0,2}^{(2)}=0,\quad
a^{(2)}_{1,1}=-2a_{2,0}^{(1)}=0.
\]
Finally,
taking $(n,m)=(1,1)$ in (\ref{eq:6}) and using (\ref{step1})  we readily deduce
\[
a^{(1)}_{0,2}=\frac{1}{2} a_{1,0}^{(2)}=0,\quad
a^{(2)}_{2,0}= \frac{1}{2} a_{1,0}^{(1)}= 0.
\]
\end{proof}

\begin{lemma}\label{lem:nm3,5}
All the coefficients $a_{n,m}^{(j)}$ in \eqref{eq:anm} with $n+m=3, 5$ can be
expressed by one parameter, respectively, namely
\begin{equation}
\label{eq:7}
\begin{split}
&a^{(1)}_{3,0}=0,\quad a^{(1)}_{2,1}=3\eta,\quad a^{(1)}_{1,2}=0,\quad\quad a^{(1)}_{0,3}=-\eta,\\
&a^{(2)}_{3,0}=\eta,\quad a^{(2)}_{2,1}=0,\quad\; a^{(2)}_{1,2}=-3\eta,\quad a^{(2)}_{0,3}=0
\end{split}
\end{equation}
and
\begin{equation}
\label{eq:9}
\begin{split}
&a^{(1)}_{5,0}=0,\; a^{(1)}_{4,1}=5\eta_1, \; a^{(1)}_{3,2}=0, \qquad a^{(1)}_{2,3}=-10\eta_1, \;
a^{(1)}_{1,4}=0,\;\; a^{(1)}_{0,5}=\eta_1\\
&a^{(2)}_{5,0}=\eta_1,\;  a^{(2)}_{4,1}=0, \;\; a^{(2)}_{3,2}=-10\eta_1,\; a^{(2)}_{2,3}=0, \quad\quad a^{(2)}_{1,4}=5\eta_1,\; a^{(2)}_{0,5}=0
\end{split}
\end{equation}
for some $\eta, \eta_1\in \mathbb R$.
\end{lemma}
%

\begin{proof}
We start with $n+m=3$. Setting $n,m=3$ in (\ref{eq:5}) and $(n,m)=(2,0),(1,1),(0,2)$ in (\ref{eq:4}),
respectively, then $(n,m)=(2,1)$ in (\ref{eq:6}), we readily get
\[
a^{(1)}_{3,0}=0,\quad a^{(2)}_{0,3}=0, \quad
a^{(2)}_{2,1}=0,\quad a^{(1)}_{2,1}=-a^{(2)}_{1,2},\quad  a^{(1)}_{1,2}=0, \quad
3a^{(2)}_{3,0}=a^{(1)}_{2,1},\quad 3a^{(1)}_{0,3}=a^{(2)}_{1,2}.
\]
%
From these relations, we can easily see that
all the coefficients $a^{(j)}_{n,m}$ with $n+m=3$
can be expressed by one parameter, say $\eta\in \mathbb R$, as in \eqref{eq:7}.

For the case $n+m=5$,
we do the same as we did above for $n+m=3$ by using the relations (\ref{eq:4}), (\ref{eq:5}) and (\ref{eq:6})
to represent all the coefficients $a^{(j)}_{n,m}$ (with $n+m=5$)
in terms of three parameters, say $\eta_1,\eta_2,\eta_3\in \R$,
%
\begin{equation}
\label{eq:8}
\begin{split}
&a^{(1)}_{5,0}=0,\;\; a^{(1)}_{4,1}=5\eta_1,\;a^{(1)}_{3,2}=\eta_3, \quad\;\;a^{(1)}_{2,3}=-10\eta_2, \;a^{(1)}_{1,4}=0,\;\quad a^{(1)}_{0,5}=\eta_2 \\
&a^{(2)}_{5,0}=\eta_1,\; a^{(2)}_{4,1}=0,\;\; a^{(2)}_{3,2}=-10\eta_1,\;a^{(2)}_{2,3}=-\eta_3,\; \quad a^{(2)}_{1,4}=5\eta_2,\;a^{(2)}_{0,5}=0.
\end{split}
\end {equation}

Next, we utilize (\ref{eq:3}) to derive a possible relation between three parameters $\eta_1,\eta_2,\eta_3$ in \eqref{eq:8}.
In fact, by setting $n=0,m=1,j=1$ in (\ref{eq:3}) and using the relations in (\ref{eq:7}) and (\ref{eq:8})
for $a_{n,m}^{(1)}$, we can deduce that
\[
4!\times5\eta_1-2\times 2\times3\times2\times10\eta_2+5!\times \eta_2+(q_1+q_2)2\times 3\eta-(q_1+q_2)3\times2 \eta=0\,,
\]
which implies $\eta_1=\eta_2$ by noting the fact
that the last two terms in the above equation cancel out.
Analogously, one can get $\eta_3=0$ by setting $n=1,m=0,j=2$ in (\ref{eq:3}).
Consequently,  the fact that $\eta_1=\eta_2, \eta_3=0$ enables us to reduce (\ref{eq:8}) to
the desired one-parameter system \eqref{eq:9}.
\end{proof}

\begin{lemma}\label{lem:2k+1}
The coefficients $a_{n,m}^{(j)}$ in \eqref{eq:anm} satisfy that
$a^{(1)}_{n,m}=a^{(2)}_{n,m}=0$ for $n+m=2k+1$, with all $k\ge1$.
\end{lemma}

\begin{proof}
With the results in Lemma \ref{lem:nm3,5} for $k=1,2$, namely $n+m=3,5$,
we may expect that the coefficients $a_{n,m}^{(j)}$ with $M:=n+m=2k+1$ for any number $k\geq 3$
can also be expressed in some parameters.
To verify this,
we write all these coefficients in terms of $M-2=2k-1$ parameters in two groups:
$c_1, c_2,\cdots, c_{\frac{M-1}{2}} $ and $b_1,b_2,\cdots, b_{\frac{M-3}{2}}$.
More precisely,  we may assume by using (\ref{eq:4})-(\ref{eq:6})  that
\be\label{eq:20}\begin{split}
&&
\begin{pmatrix}
a_{M,0}^{(1)} \\
a_{M,0}^{(2)}
\end{pmatrix}=
\begin{pmatrix}
 0\\ c_1
\end{pmatrix},\quad
\begin{pmatrix}
a_{M-1,1}^{(1)} \\
a_{M-1,1}^{(2)}
\end{pmatrix}=\begin{pmatrix}
Mc_1 \\
0
\end{pmatrix},\quad
\begin{pmatrix}
a_{M-2,2}^{(1)} \\
a_{M-2,2}^{(2)}
\end{pmatrix}=
\begin{pmatrix}
b_1 \\
-\frac{M(M-1)}{2}c_1
\end{pmatrix}, \hskip1truecm
\\
&&\begin{pmatrix}
a_{M-3,3}^{(1)} \\
a_{M-3,3}^{(2)}
\end{pmatrix}=
\begin{pmatrix}
c_2 \\
-\frac{M-2}{3}b_1
\end{pmatrix},\qquad\cdots,\qquad
\begin{pmatrix}
a_{3,M-3}^{(1)} \\
a_{3,M-3}^{(2)}
\end{pmatrix}=\begin{pmatrix}
 b_{\frac{M-3}{2}} \\ -\frac{4}{M-3}c_{\frac{M-3}{2}}
 \end{pmatrix}, \hskip1truecm
 \\
&&\begin{pmatrix}
a_{2,M-2}^{(1)} \\
a_{2,M-2}^{(2)}
\end{pmatrix}=\begin{pmatrix}
 -\frac{M(M-1)}{2}c_{\frac{M-1}{2}}\\ -\frac{3}{M-2}b_{\frac{M-3}{2}}
\end{pmatrix},\quad
\begin{pmatrix}
a_{1,M-1}^{(1)} \\
a_{1,M-1}^{(2)}
\end{pmatrix}=\begin{pmatrix}
 0\\ M c_{\frac{M-1}{2}}
\end{pmatrix},\quad
\begin{pmatrix}
a_{0,M}^{(1)} \\
a_{0,M}^{(2)}
\end{pmatrix}=\begin{pmatrix}
 c_{\frac{M-1}{2}}\\ 0
\end{pmatrix}.\end{split}
\en
Then the desired results in Lemma\,\ref{lem:2k+1} are a consequence of the following two steps,
first to show all $c_j$ in \eqref{eq:20} vanish (Step 1), then to prove the same for all $b_j$ above (Step 2).

{\bf Step 1}: Prove that $c_j=0$ for all $j=1,\cdots,(M-1)/2$.

For the purpose of induction, we rewrite the relations in (\ref{eq:9}) (that is, $k=2$ or $M=5$) involving the non-vanishing parameters
in a more general form as follows:
\begin{equation}
\label{eq:22}
\begin{split}
&a^{(1)}_{M-1,1}=M\eta_{M}, ~a^{(1)}_{M-3,3}=-\frac{M!}{(M-3)!3!}\eta_{M},\; \cdots ,
~a^{(1)}_{4,M-4}=(-1)^{\frac{M-5}{2}}\frac{M!}{(M-4)!4!}\eta_{M}\textcolor{rot}{,}\\
&a^{(1)}_{2,M-2}=(-1)^{\frac{M-3}{2}}\frac{M(M-1)}{2}\eta_{M},
~a^{(2)}_{M-2,2}=-\frac{M(M-1)}{2}\eta_{M}, ~a^{(1)}_{0,M}=(-1)^{\frac{M-1}{2}}\eta_M,
\\
&a^{(2)}_{M,0}=\eta_{M},
~a^{(2)}_{M-4,4}=\frac{M!}{(M-4)!4!}\eta_M,
~a^{(2)}_{M-6,6}=-\frac{M!}{(M-6)!6!}\eta_M,\cdots \,
\\ &a^{(2)}_{3,M-3}=(-1)^{\frac{M-3}{2}}\frac{M!}{(M-3)!3!}\eta_{M},
~a^{(2)}_{1,M-1}=(-1)^{\frac{M-1}{2}}M\eta_M. 
\end{split}
\end{equation}
Note that the constant $\eta_1$ in (\ref{eq:9}) has been replaced by $\eta_{M}$ and that
the coefficients before $\eta_M$ in the above relations are derived from (\ref{eq:4})-(\ref{eq:6}). This confirms that for $M=5$ (or $k=2$), the coefficients
 in (\ref{eq:20}) related to $c_j$ $(j=1,2,\cdots,c_{\frac{M-1}{2}})$ depend only on $\eta_M$.
For any fixed $k\geq 3$, i.e., $M\geq 7$, we now verify all the relations
in (\ref{eq:22}) under the induction hypothesis that
\begin{equation}
\begin{split}
&a^{(1)}_{M^{'}-1,1}=M^{'}\eta_{M^{'}}, ~a^{(1)}_{M^{'}-3,3}=-\frac{M^{'}!}{(M^{'}-3)!3!}\eta_{M^{'}},\cdots \,,
~a^{(1)}_{4,M^{'}-4}=(-1)^{\frac{M^{'}-5}{2}}\frac{M^{'}!}{(M^{'}-4)!4!}\eta_{M^{'}}
\\ &a^{(1)}_{2,M^{'}-2}=(-1)^{\frac{M^{'}-3}{2}}\frac{M^{'}(M^{'}-1)}{2}\eta_{M^{'}},
~a^{(1)}_{0,M^{'}}=(-1)^{\frac{M^{'}-1}{2}}\eta_{M^{'}}\\
&a^{(2)}_{M^{'},0}=\eta_{M^{'}}, ~a^{(2)}_{M^{'}-2,2}=-\frac{M^{'}(M^{'}-1)}{2}\eta_{M^{'}},
~a^{(2)}_{M^{'}-4,4}=\frac{M^{'}!}{(M^{'}-4)!4!}\eta_{M^{'}}, \cdots \,\\ &a^{(2)}_{3,M^{'}-3}=(-1)^{\frac{M^{'}-3}{2}}\frac{M^{'}!}{(M^{'}-3)!3!}\eta_{M^{'}},
~a^{(2)}_{1,M^{'}-1}=(-1)^{\frac{M^{'}-1}{2}}M^{'}\eta_{M^{'}},
\end{split}
\end{equation}
for all $M'=2k'+1$, $k'=0,1,\cdots, k-1$.
This implies that all the coefficients in (\ref{eq:20}) related to $c_j$ depend actually on one parameter.
For this purpose, it suffices to verify that
 the coefficients $a_{n,m}^{(j)}$ in (\ref{eq:22})  satisfying $n+m=M$ can be represented
by the unique parameter $\eta_{M}$.
To do so, we take $n=M-4,m=0,j=2$ and $n=M-5,m=1,j=1$ in (\ref{eq:3}), respectively,
and utilize the induction hypothesis above to come readily to the relations
\begin{eqnarray}
&&\Big\{\frac{M!}{(M-4)!}a^{(2)}_{M,0}-2\frac{M!}{(M-4)!}\Big\}c_1-6(M-3)c_2+q_1q_2\eta_{M-4}=0,
\label{eq:26}\\
&&\frac{M!}{(M-5)!}c_1+12(M-3)(M-4) c_2+ 120 c_3 +(M-4)q_1q_2\eta_{M-4}=0.
\label{eq:27}
\end{eqnarray}

%
Analogously, setting $j=1$, $n=M-2k-5$ and $m=2k+1$ for all $1< k \le {(M-5)}/{2}$  in (\ref{eq:3}),
we obtain in combination of  (\ref{eq:26}) and (\ref{eq:27}) the following linear system
\be\label{ls}
A_M\,X_M=B_M,
\en
where $A_M$, $B_M$ and $X_M$ are given by
\[
A_M=\begin{pmatrix}
\begin{smallmatrix}
-\frac{M!}{(M-4)!} & -(M-3)3\times2 &0 &\dots &0 &0 \\
\frac{M!}{(M-5)!}  & 2(M-3)(M-4)3\times2 &5\times4\times3\times2 &0&\dots &0 \\
0 &(M-3)(M-4)(M-5)(M-6) &2(M-5)(M-6)5\times4 &7\times6\times5\times4 &0 & 0 \\
0 & 0 & \ddots & \ddots & \ddots  &\vdots\\
0 &\dots &0  &6\times5\times4\times3 & 2(M-4)(M-5)3\times4 &-\frac{\frac{(M)!}{(M-6)!}}2 \\
0 &\dots &0 &0 &4\times3\times2 & -\frac{M!}{(M-4)!}
\end{smallmatrix}
\end{pmatrix},
\]

\[
B_M=\begin{pmatrix}
\begin{smallmatrix}
&\eta_{M-4}\\
&(M-4)\eta_{M-4}\\
&-\frac{(M-4)!}{(M-7)!3!}\eta_{M-4}\\
&\vdots\\
&(-1)^{\frac{M-7}{2}}\frac{(M-4)(M-5)}{2}\eta_{M-4}\\
&(-1)^{\frac{M-5}{2}}\eta_{M-4}
\end{smallmatrix}
\end{pmatrix},
\quad
X_M=\begin{pmatrix}
\begin{smallmatrix}
&c_1\\
&c_2\\
&\vdots\\
&c_{\frac{M-1}2}
\end{smallmatrix}
\end{pmatrix}\,.
\]
We next demonstrate $\eta_{M-4}=0$ by diagonalizing the matrix $A_M$. To this aim,
we form the augmented matrix $\widetilde{A}=[A_M, B_M]$:
\[
\begin{pmatrix}
\begin{smallmatrix}
-\frac{M!}{(M-4)!} & -6(M-3) &0 &\dots &0 &0 &\eta_{M-4} \\
\frac{M!}{(M-5)!}  & 12(M-3)(M-4) &5\times4\times3\times2 &0&\dots &0 &(M-4)\eta_{M-4} \\
0 &\frac{(M-3)!}{(M-7)!} &2(M-5)(M-6)5\times4 &7\times6\times5\times4 &0 & 0  &-\frac{(M-4)!}{(M-7)!3!}\eta_{M-4}\\
0 & 0 & \ddots & \ddots & \ddots  &\vdots& \vdots \\
0 &\dots &0  &6\times5\times4\times3 & 2(M-4)(M-5)3\times4 &-\frac{(M)!}{2(M-6)!} &(-1)^{\frac{M-7}{2}}\frac{(M-4)(M-5)}{2}\eta_{M-4}\\
0 &\dots &0 &0 &4\times3\times2 & -\frac{M!}{(M-4)!}&(-1)^{\frac{M-5}{2}}\eta_{M-4}.
\end{smallmatrix}
\end{pmatrix}.
\]
Below we shall often write as $r_j$ the $j$-th row of the matrix $\widetilde{A}$ or the transformed variant
of $\widetilde{A}$ by an elementary transformation.
By applying $r_1\times(M-4)+r_2$ to $\widetilde{A}$, we have the matrix $\widetilde{A}_1$:
\[
\begin{pmatrix}
\begin{smallmatrix}
-\frac{M!}{(M-4)!} & -6(M-3) &0 &\dots &0 &0 &\eta_{M-4} \\
0  & 6(M-3)(M-4) &5\times4\times3\times2 &0&\dots &0 &2(M-4)\eta_{M-4} \\
0 &\frac{(M-3)!}{(M-7)!} &2(M-5)(M-6)5\times4 &7\times6\times5\times4 &0 & 0  &-\frac{(M-4)!}{(M-7)!3!}\eta_{M-4}\\
0 & 0 & \ddots & \ddots & \ddots  &\vdots& \vdots \\
0 &\dots &0  &6\times5\times4\times3 & 2(M-4)(M-5)3\times4 &-\frac{(M)!}{2(M-6)!} &(-1)^{\frac{M-7}{2}}\frac{(M-4)(M-5)}{2}\eta_{M-4}\\
0 &\dots &0 &0 &4\times3\times2 & -\frac{M!}{(M-4)!}&(-1)^{\frac{M-5}{2}}\eta_{M-4}
\end{smallmatrix}
\end{pmatrix}\,,
\]
to which we apply the transformation $-\frac{r_2\times(M-5)(M-6)}6+r_3$ to get
the matrix $\widetilde{A}_2$:
\[
\begin{pmatrix}
\begin{smallmatrix}
-\frac{M!}{(M-4)!} & -6(M-3) &0 &\dots &0 &0 &\eta_{M-4} \\
0  & 6(M-3)(M-4) &5\times4\times3\times2 &0&\dots &0 &2(M-4)\eta_{M-4} \\
0 &0 &20(M-5)(M-6) &7\times6\times5\times4 &0 & 0  &-3\frac{(M-4)!}{(M-7)!3!}\eta_{M-4}\\
0 & 0 & \ddots & \ddots & \ddots  &\vdots& \vdots \\
0 &\dots &0  &6\times5\times4\times3 & 2(M-4)(M-5)3\times4 &-\frac{(M)!}{2(M-6)!} &(-1)^{\frac{M-7}{2}}\frac{(M-4)(M-5)}{2}\eta_{M-4}\\
0 &\dots &0 &0 &4\times3\times2 & -\frac{M!}{(M-4)!}&(-1)^{\frac{M-5}{2}}\eta_{M-4}
\end{smallmatrix}
\end{pmatrix}.
\]
Repeating the above process, we come to the matrix $\widetilde{A}_{\frac{M-3}{2}-1}$:
\[
\begin{pmatrix}
\begin{smallmatrix}
-\frac{M!}{(M-4)!} & -6(M-3) &0 &\dots &0 &0 &\eta_{M-4} \\
0  & 6(M-3)(M-4) &5\times4\times3\times2 &0&\dots &0 &2(M-4)\eta_{M-4} \\
0 &0 &20(M-5)(M-6) &7\times6\times5\times4 &0 & 0  &-3\frac{(M-4)!}{(M-7)!3!}\eta_{M-4}\\
0 & 0 & \ddots & \ddots & \ddots  &\vdots& \vdots \\
0 &\dots &0  &0 & (M-4)(M-5)3\times4 &-\frac{(M)!}{2(M-6)!} &(\frac{M-3}{2})(-1)^{\frac{M-7}{2}}\frac{(M-4)(M-5)}{2}\eta_{M-4}\\
0 &\dots &0 &0 &4\times3\times2 & -\frac{M!}{(M-4)!}&(-1)^{\frac{M-5}{2}}\eta_{M-4}
\end{smallmatrix}
\end{pmatrix}.
\]
Now, taking the action $- r_{\frac{M-3}{2}} \times \frac 2 {(M-4)(M-5)}+r_{\frac{M-1}{2}}$,
we may transform $\widetilde{A}_{\frac{M-3}{2}-1}$ into a new matrix whose last row is given by
\ben
\begin{pmatrix}
0 & \cdots & 0 &0 & 0 & 0 & (-1)^{\frac{M-5}{2}}(\eta_{M-4}+(\frac{M-3}{2})\eta_{M-4})
\end{pmatrix}\,.
\enn
This, along with the linear system (\ref{ls}), leads to the relation
\[
(-1)^{\frac{M-5}{2}}(\eta_{M-4}+(\frac{M-3}{2})\eta_{M-4})=0,
\]
from which we see that
\be\label{eq: M}
\eta_{M-4}=0.
\en

%
Now, the nonhomogeneous system \eqref{ls}
reduces to the homogeneous one $A_M\,X_M=0$ because of the result \eqref{eq: M}.
We can easily trace from the previous linear transformations we have applied to $A_M$
to know that
\begin{equation}
\label{eq:28}
\mbox{rank}(A_M)=\frac{M-1}{2}-1.
\end{equation}
This means that the non-trivial solutions to $A_M X_M=0$ forms a one-dimensional space.
Then we can directly derive from \eqref{ls} by taking $c_1=\eta_{M}$ as a single parameter that
\begin{equation}
\begin{pmatrix}
c_1\\
c_2\\
\vdots\\
c_{\frac{M-3}2}\\
c_{\frac{M-1}2}\\
\end{pmatrix}
=
\begin{pmatrix}
\eta_{M}\\
-\frac{M!}{(M-3)!3!}\eta_{M}\\
\vdots\\
(-1)^{\frac{M-5}{2}}\frac{M!}{(M-4)!4!}\eta_{M}\\
(-1)^{\frac{M-1}{2}}\eta_{M}\,,
\end{pmatrix}
\end{equation}
hence prove the relations in (\ref{eq:22}) for all $M\geq 7$.
By the arbitrariness of $M=2k+1$ for $k\geq 3$ and induction argument, we can conclude
that $\eta_M=0$ for all odd integers $M\in \N^+$, and in particular, the vanishing
of $c_j$, $j=1,2,\cdots, c_{{(M-1)}/{2}}$.

{\bf Step 2}: Prove that $b_j=0$ for all $j=1,2,\cdots, {(M-3)}/{2}$, and all $M=2k+1$ with $k\geq 3$.
This yields the vanishing of the coefficients $a_{n,m}^{(j)}$ that depend
on $b_1,b_2,\cdots, b_{{(M-3)}/{2}}$ in (\ref{eq:20}).
%

Again, we shall use the induction argument to prove that
\begin{equation}
\label{eq:31}
\begin{split}
&a^{(1)}_{M,0}=a^{(1)}_{M-2,2}=\cdots \,=a^{(1)}_{3,M-3}=a^{(1)}_{1,M-1}=0,\\
&a^{(2)}_{M-1,1}=0=a^{(2)}_{M-3,3}=\cdots \,=a^{(2)}_{2,M-2}=a^{(2)}_{0,M}=0,
\end{split}
\end{equation}
for all $M=2k+1$, $k\geq 3$. Note that the case of $k=2$ (or $M=5$) was verified in (\ref{eq:9}).
We make the induction hypothesis that
\[
\begin{split}
&a^{(1)}_{M^{'},0}=a^{(1)}_{M^{'}-2,2}=\cdots \,=a^{(1)}_{3,M^{'}-3}=a^{(1)}_{1,M^{'}-1}=0\\
&a^{(2)}_{M^{'}-1,1}=0=a^{(2)}_{M^{'}-3,3}=\cdots \,=a^{(2)}_{2,M^{'}-2}=a^{(2)}_{0,M^{'}}=0,
\end{split}
\]
for all $M'=2k'+1$, $0\le k'<k$.
Then by setting $j=1$, $n=M-4,m=0$ and $n=M-6,m=2$ in (\ref{eq:3}), respectively,
and using the induction hypothesis, we obtain that
\begin{eqnarray}
 4\times(M-2)(M-3) b_{1}+4! b_{2}=0\,, \label{eq:17+}\\
\frac{(M-2)!}{(M-6)!}b_{1}+2(M-4)(M-5)4\times3b_{2}+6\times5\times4\times3b_{3}=0.
\label{eq:18+}
\end{eqnarray}
%
We can continue this process, by taking $j=1$, $n=M-2k-4$ and $m=2k$
for any $2\le k \le \frac{M-5}{2}$ in (\ref{eq:3}),
then using (\ref{eq:17+}) and (\ref{eq:18+}), to derive
the homogeneous linear algebraic system
\[
\tilde{G}_M Y_M=0,
\]
where
{\small
\[
\tilde{G}_M=\begin{pmatrix}
\begin{smallmatrix}
4\times(M-2)(M-3) &4! &\dots &0 &0\\
(M-2)(M-3)(M-4)(M-5) &2(M-4)(M-5)4\times3 &6\times5\times4\times3&\dots &0\\
0&(M-4)(M-5)(M-6)(M-7)&2(M-6)(M-7)6\times5 &8\times7\times6\times5&0 \\
0& \ddots & \ddots & \ddots  &\vdots \\
\dots &0  &7\times6\times5\times4 & 2(M-5)(M-6)5\times4 &\frac{(M-3)!}{(M-7)!}\\
\dots &0 &0 &5\times4\times3\times2 & 2(M-3)(M-4)6
\end{smallmatrix}
\end{pmatrix}
\]
}
and
\[
Y_M=\begin{pmatrix}
b_1,&
b_2,&
\cdots \,&
b_{\frac{M-3}2}&
\end{pmatrix}'.
\]
It is easy to find through diagonalization that $\mbox{Det}(G_M) \ne0$, hence leading to
the vanishing of $b_j$, $j=1,2,\cdots, {(M-3)}/{2}$.
\end{proof}

\begin{lemma}\label{lem:step4}
The coefficients $a_{n,m}^{(j)}$ in \eqref{eq:anm} satisfy that
$a^{(1)}_{n,m}=a^{(2)}_{n,m}=0$ for $n+m=2k$, with all $k\ge2$.
%
\end{lemma}

\begin{proof}
The argument is carried out in a similar manner to the one for Lemma\,\ref{lem:2k+1}.
We first show that $a^{(1)}_{n,m}=a^{(2)}_{n,m}=0$ for $k=2$, or $n+m=4$. To this aim, we set $n=m=4$ in (\ref{eq:5})
and $n=m=3$ in (\ref{eq:6}), respectively, to sfind that
\[
a^{(1)}_{4,0}=a^{(2)}_{0,4}=0 \quad \mbox{and} \quad
4a^{(2)}_{4,0}=a^{(1)}_{3,1},\quad 4a^{(1)}_{0,4}=a^{(2)}_{1,3}.
\]
Then setting $n=n_j,m=m_j,n_j+m_j=3$ for $j=3,2,1,0$ in (\ref{eq:4}), respectively, we easily derive
\[
a^{(2)}_{3,1}=0, \quad \frac32a^{(1)}_{3,1}=-a^{(2)}_{2,2},\quad \frac23 a^{(1)}_{2,2}=-a^{(2)}_{1,3},\quad a^{(1)}_{1,3}=0.
\]
Therefore, we are able to express all the coefficients $a_{n,m}^{(j)}$ with $n+m=4$
in two parameters $\eta_1, \eta_2\in \mathbb R$ as follows:
\be\label{eq:eta}
\begin{split}
&a^{(1)}_{4,0}=0, ~a^{(1)}_{3,1}=4\eta_1,a^{(1)}_{2,2}=-6\eta_2, ~a^{(1)}_{1,3}=0, ~a^{(1)}_{0,4}=\eta_2, \\
&a^{(2)}_{4,0}=\eta_1, ~a^{(2)}_{3,1}=0, ~a^{(2)}_{2,2}=-6\eta_1, ~a^{(2)}_{1,3}=4\eta_2, ~a^{(2)}_{0,4}=0.\\
\end{split}
\en
Furthermore, by taking $n=0,m=0$ in (\ref{eq:3}), we obtain
\[
4\, !\, a_{0,4}^{(j)} + 4\, !\, a_{4,0}^{(j)} +8 a_{2,2}^{(j)}=0,\quad j=1,2,
\]
which, along with (\ref{eq:eta}),
concludes that $\eta_1=\eta_2=0$. Therefore, we have shown that $a^{(j)}_{n,m}=0$ for all $n+m=4$ and $j=1,2$.

For any fixed $k\geq 3$, we make the induction hypothesis that
\[
a_{n,m}^{(j)}=0\quad\mbox{for all}\quad M= n+m=2k^{'},\; ~0\le k^{'}<k.
\]
\par
Then we argue analogously to what we did for Lemma\,\ref{lem:2k+1}
to derive from (\ref{eq:4})-(\ref{eq:6}) that
\be\label{eq:hy}
\begin{split}
&a^{(1)}_{M-1,1}=Mc_1, ~a^{(1)}_{M-3,3}=c_2,\; ~\cdots \,,  ~a^{(1)}_{3,M-3}=c_{\frac{M}2-1} , ~a^{(1)}_{1,M-1}=0,\\
&a^{(2)}_{M,0}=c_1, ~a^{(2)}_{M-2,2}=-\frac{M(M-1)}{2}c_1,\; \cdots \,,  ~a^{(2)}_{2,M-2}=-c_{\frac{M}2-1},
~a^{(2)}_{0,M}=0,\\
\end{split}
\en
and
\begin{equation}
\begin{split}
&a^{(1)}_{M,0}=0, ~a^{(1)}_{M-2,2}=b_1,\; \cdots, \, ~a^{(1)}_{2,M-2}=-\frac{M(M-1)}{2}b_{\frac{M}2-1}, ~a^{(1)}_{0,M}=b_{\frac{M}2-1},\\
&a^{(2)}_{M-1,1}=0, \quad a^{(2)}_{M-3,3}=\cdots \,= a^{(2)}_{1,M-1}=Mb_{\frac{M}2-1},\\
\end{split}
\end{equation}
where $c_1$, $c_2$,$\cdots$, $c_{{M}/{2}-1}$ and
$b_1$, $b_2$, $\cdots$, $b_{{M}/{2}-1}$ are all constants in $\C$.

Next, we show that all these constants are identically zero.
To prove that all the constants $c_j$ for $j=1,2,\cdots, M/2-1$ are zero, we
set $n=M-4,m=0,j=2$ and $n=M-5,m=1,j=1$ in (\ref{eq:3}), respectively,
and use the induction hypothesis (\ref{eq:hy}) to deduce
that
\begin{equation}
\label{eq:14}
[\frac{M!}{(M-4)!}a^{(2)}_{M,0}-2\frac{M!}{(M-4)!}]c_1-(M-3)3\times 2c_2=0\,,
\end{equation}
\begin{equation}
\label{eq:15}
\frac{M!}{(M-5)!}c_1+2(M-3)(M-4)3\times2c_2+5\times4\times3\times2c_3=0.
\end{equation}
We can repeat this process by taking $j=1$ and $n=M-2k-5$,$m=2k+1$ for all $0< k \le \frac{M-6}{2}$ in (\ref{eq:3})
to arrive at the linear system
\[
A\, X=0,
\]
where
\[
A:=\begin{pmatrix}
\begin{smallmatrix}
-\frac{M!}{(M-4)!} & -(M-3)3\times2 &0 &\dots &0 &0\\
\frac{M!}{(M-5)!}  & 2(M-3)(M-4)3\times2 &5\times4\times3\times2 &0&\dots &0\\
0 &(M-3)(M-4)(M-5)(M-6) &2(M-5)(M-6)5\times4 &7\times6\times5\times4 &0 &0 \\
0 & 0 & \ddots & \ddots & \ddots  &\vdots \\
0 &\dots &0  &7\times6\times5\times4 & 2(M-5)(M-6)5\times4 &\frac{(M-3)!}{(M-7)!}\\
0 &\dots &0 &0 &5\times4\times3\times2 & 2(M-3)(M-4)3\times2 &
\end{smallmatrix}
\end{pmatrix},
\]
and
\[
X=\begin{pmatrix}
c_1,&
c_2,&
\cdots \,&
c_{\frac{M}2-1}&
\end{pmatrix}^{'}.
\]
Again, by a diagonlization process we can directly verify
that $\mbox{Det} (A)\ne0$, implying that
\[
c_1=c_2=\cdots \,=c_{\frac{M}2-1}=0.
\]

It remains to show that all constants $b_j$ for $j=1,2,\cdots,M/2-1$ also vanish.
For this, we set $n=0,m=M-4,j=1$ and  $n=2,m=M-6,j=1$ in (\ref{eq:3}) respectively to see
that
\begin{equation}
\label{eq:17}
[\frac{M!}{(M-4)!}-2\frac{M!}{(M-4)!}]b_{\frac{M}{2}-1}+4\times3\times b_{\frac{M}{2}-2}=0\,,
\end{equation}
\begin{equation}
\label{eq:18}
-\frac{M!}{2(M-6)!}b_{\frac{M}{2}-1}+2(M-4)(M-5)4\times3b_{\frac{M}{2}-2}+6\times5\times4\times3b_{{\frac{M}{2}}-3}=0.
\end{equation}
Then we may continue this process by taking $j=1$ and $n=2k$,$m=M-2k-4$ for all $2\leq k \leq \frac{M-4}{2}$ in (\ref{eq:3}), to come up with the linear system
\ben
\tilde{A}\; Y=0,
\enn
where
\[
\tilde{A}:=\begin{pmatrix}
\begin{smallmatrix}
-\frac{M!}{(M-4)!} & 4\times3\times2 &0 &\dots &0 &0\\
-\frac{\frac{M!}{(M-6)!}}2  & 2(M-4)(M-5)4\times3 &6\times5\times4\times3 &0&\dots &0\\
0 &(M-4)(M-5)(M-6)(M-7) &2(M-6)(M-7)6\times5 &8\times7\times6\times5 &0 &0 \\
0 & 0 & \ddots & \ddots & \ddots  &\vdots \\
0 &\dots &0  &6\times4\times3\times2 & 2(M-4)(M-5)4\times3 &\frac{(M-2)!}{(M-6)!}\\
0 &\dots &0 &0 &4\times3\times2 & 4(M-2)(M-3) &
\end{smallmatrix}
\end{pmatrix},
\]
and
\[
Y=\begin{pmatrix}
b_{\frac{M}2-1},&
b_{\frac{M}2-2},&
\cdots \,&
b_1
\end{pmatrix}^{'}.
\]
As we did earlier, we can verify that $\mbox{Det} (\tilde{A})\ne0$, therefore derive the desired
results that
\[
b_1=b_2=\cdots =b_{\frac{M}2-1}=0.
\]
\end{proof}

Now the result of Lemma \ref{Lem:1} follows directly from Lemmas \ref{lem:step2}-\ref{lem:step4}.

\begin{remark}
 We believe that it might be possible to prove Theorem \ref{TH} in a general planar corner domain whose angle lies in $(0,2\pi)\backslash\{\pi\}$. However, as one could expect, much more complicated technicalities
 will be involved in the proof of the analogue of Lemma \ref{Lem:1} by our approach.
 The inverse transmission problem for the scalar Helmholtz equation (\ref{TE-H})-(\ref{TE-Ht}) with such general angles also deserves further investigation.
\end{remark}

\section{Acknowledgments}

The work of G. Hu is supported by the NSFC grant (No. 11671028) and NSAF grant (No. U1530401).
The work of L. Li is partially supported by the National Science Foundation of
China (61421062, 61520106004) and Microsoft Research of Asia.
The work of J. Zou was supported by the
Hong Kong RGC General Research Fund (project 14322516) and National Natural Science
Foundation of China/Hong Kong Research Grants Council Joint Research Scheme 2016/17 (project
N\_CUHK437/16).

\end{document}